\documentclass[leqno,10pt,a4paper]{amsart}
\usepackage[full]{textcomp}
\usepackage{newpxtext}
\usepackage{cabin} % sans serif
\usepackage[varqu,varl]{inconsolata} % sans serif typewriter
\usepackage[bigdelims,vvarbb]{newpxmath} % bb from STIX
\usepackage[cal=cm,bb=esstix,bbscaled=1.05]{mathalfa} % mathcal
\usepackage[margin=2.4cm]{geometry}
\usepackage[usenames,dvipsnames]{color}
\usepackage{cite}
\usepackage{hyperref}
\hypersetup{
 colorlinks,
 linkcolor={red!50!black},
 citecolor={blue!50!black},
 urlcolor={blue!80!black}
}
\usepackage{tikz}
\usepackage{graphicx}
\usepackage{mleftright}
\usepackage{booktabs}
\usepackage{cite}
\usepackage{enumitem}
\setlist[enumerate]{labelsep=*, leftmargin=1.5pc}
\setlist[enumerate]{label=\normalfont(\roman*), ref=\roman*}
\newtheorem{thm}{Theorem}[section]
\newtheorem{lemma}[thm]{Lemma}
\newtheorem{cor}[thm]{Corollary}

\theoremstyle{definition}
\newtheorem{example}[thm]{Example}
\newtheorem{remark}[thm]{Remark}
\newtheorem{definition}[thm]{Definition}

\numberwithin{equation}{section}
\renewcommand{\emptyset}{\varnothing}
\newcommand{\Z}{\mathbb{Z}}
\newcommand{\Q}{\mathbb{Q}}
\newcommand{\C}{\mathbb{C}}
\newcommand{\A}{\mathbb{A}}
\newcommand{\T}{\mathbb{T}}
\newcommand{\NQ}{N_\Q}
\newcommand{\MQ}{M_\Q}
\newcommand{\Proj}{\mathbb{P}}
\newcommand{\abs}[1]{\left\vert{#1}\right\vert}
\newcommand{\dual}[1]{{#1}^*}
\newcommand{\orig}{\textbf{0}}
\newcommand{\bdry}[1]{\partial{#1}}
\newcommand{\intr}[1]{{#1}^\circ}
\newcommand{\Vol}[1]{\operatorname{Vol}\mleft({#1}\mright)}
\newcommand{\V}[1]{\operatorname{vert}\mleft({#1}\mright)}
\renewcommand{\dim}[1]{\operatorname{dim}\mleft({#1}\mright)}
\renewcommand{\gcd}[1]{\operatorname{gcd}\mleft\{{#1}\mright\}}
\newcommand{\lcm}[1]{\operatorname{lcm}\mleft\{{#1}\mright\}}
\newcommand{\blcm}[1]{\operatorname{lcm}\mleft({#1}\mright)}
\renewcommand{\min}[1]{\operatorname{min}\mleft\{{#1}\mright\}}

\newcommand{\conv}[1]{\operatorname{conv}\mleft({#1}\mright)}
\newcommand{\sconv}[1]{\operatorname{conv}\mleft\{{#1}\mright\}}
\newcommand{\scone}[1]{\operatorname{cone}\mleft\{{#1}\mright\}}
\newcommand{\umin}{u_{\operatorname{min}}}
\newcommand{\modb}[1]{\ \ \mleft(\operatorname{mod}\ {#1}\mright)}
\newcommand{\cB}{\mathcal{B}}
\newcommand{\cRB}{\mathcal{RB}}
\newcommand{\cIB}{\mathcal{IB}}
\newcommand{\cT}{\mathcal{T}}
\newcommand{\res}[1]{\operatorname{res}\mleft({#1}\mright)}
\newcommand{\ol}{\overline}
\DeclareMathOperator{\GL}{GL}
\DeclareMathOperator{\mut}{\mu}
\DeclareMathOperator{\Hom}{Hom}
\DeclareMathOperator{\Ehr}{Ehr}

\DeclareMathOperator{\Hilb}{Hilb}
\graphicspath{{images/}}
\begin{document}
%-------------------------------------------------------------------------------
\author[A.\,M.\,Kasprzyk]{Alexander~M.~Kasprzyk}
\address{School of Mathematical Sciences\\University of Nottingham\\Nottingham\\NG7 2RD\\UK}
\email{a.m.kasprzyk@nottingham.ac.uk}
\author[B.\,Wormleighton]{Ben~Wormleighton}
\address{Department of Mathematics\\University of California at Berkeley\\Berkeley\\CA\\94720\\USA}
\email{b.wormleighton@berkeley.edu}
%-------------------------------------------------------------------------------
%-------------------------------------------------------------------------------
\title{Quasi-period collapse for duals to Fano polygons:\\an explanation arising from algebraic geometry}
\maketitle
%-------------------------------------------------------------------------------
\begin{abstract}
The Ehrhart quasi-polynomial of a rational polytope $P$ is a fundamental invariant counting lattice points in integer dilates of $P$. The quasi-period of this quasi-polynomial divides the denominator of $P$ but is not always equal to it: this is called quasi-period collapse. Polytopes experiencing quasi-period collapse appear widely across algebra and geometry, and yet the phenomenon remains largely mysterious. By using techniques from algebraic geometry -- $\Q$-Gorenstein deformations of orbifold del~Pezzo surfaces -- we explain quasi-period collapse for rational polygons dual to Fano polygons and describe explicitly the discrepancy between the quasi-period and the denominator.
\end{abstract}
%-------------------------------------------------------------------------------
\section{Introduction}\label{sec:intro}
%-------------------------------------------------------------------------------
Let $P\subset\Z^d\otimes_\Z\Q$ be a convex lattice polytope of dimension $d$. Let $L_P(k):=\abs{kP\cap\Z^d}$ count the number of lattice points in dilations $kP$ of $P$, $k\in\Z_{\geq 0}$. Ehrhart~\cite{e67} showed that $L_P$ can be written as a degree $d$ polynomial
\[
L_P(k)=c_dk^d+\ldots+c_1k+c_0
\]
which we call the \emph{Ehrhart polynomial} of $P$. The leading coefficient $c_d$ is given by $\Vol{P}/d!$, $c_{d-1}$ is equal to $\Vol{\bdry P}/2(d-1)!$, and $c_0=1$. Here $\Vol{\,\cdot\,}$ denotes the normalised volume, and $\bdry P$ denotes the boundary of $P$. For example, if $P$ is two-dimensional (that is, $P$ is a lattice \emph{polygon}) we obtain
\[
L_P(k)=\frac{\Vol{P}}{2}k^2+\frac{\abs{\bdry P\cap\Z^2}}{2}k+1.
\]
Setting $k=1$ in this expression recovers Pick's Theorem~\cite{Pick}. The values of the Ehrhart polynomial of $P$ form a generating function $\Ehr_P(t):=\sum_{k\geq0}L_P(k)t^k$ called the \emph{Ehrhart series} of $P$.

When the vertices of $P$ are rational points the situation is more interesting. Recall that a \emph{quasi-polynomial} with \emph{period} $s\in\Z_{>0}$ is a function $q:\Z\rightarrow\Q$ defined by polynomials $q_0,q_1,\ldots,q_{s-1}$ such that
\[
q(k)=q_i(k)\qquad\text{when }k\equiv i\modb{s}.
\]
The \emph{degree} of $q$ is the largest degree of the $q_i$. The minimum period of $q$ is called the \emph{quasi-period}, and necessarily divides any other period $s$. Ehrhart showed that $L_P$ is given by a quasi-polynomial of degree $d$, which we call the \emph{Ehrhart quasi-polynomial} of $P$. Let $\pi_P$ denote the quasi-period of $P$. The smallest positive integer $r_P\in\Z_{>0}$ such that $r_PP$ is a lattice polytope is called the \emph{denominator} of $P$. It is certainly the case that $L_P$ is $r_P$-periodic, however it is perhaps surprising that the quasi-period of $L_P$ does not always equal $r_P$; this phenomenon is called \emph{quasi-period collapse}.

\begin{example}[Quasi-period collapse]\label{eg:quasi_period_collapse_dual_P114}
Consider the triangle $P:=\sconv{(5, -1),(-1, -1),(-1, 1/2)}$ with denominator $r_P=2$. This has $L_P(k)=9/2k^2 + 9/2k + 1$, hence $\pi_P=1$.
\end{example}

Quasi-period collapse is poorly understood, although it occurs in many contexts. For example, de~Loera--McAllister~\cite{dlm04,dlm06} consider polytopes arising naturally in the study of Lie algebras (the Gel'fand--Tsetlin polytopes and the polytopes determined by the Clebsch--Gordan coefficients) that exhibit quasi-period collapse. In dimension two McAllister--Woods~\cite{mw05} show that there exist rational polygons with $r_P$ arbitrarily large but with $\pi_P=1$ (see also Example~\ref{eg:arbitrarily_large_denominator}). Haase--McAllister~\cite{hm08} give a constructive view of this phenomena in terms of \emph{$\GL_d(\Z)$-scissor congruence}; here a polytope is partitioned into pieces that are individually modified via $\GL_d(\Z)$ transformation and lattice translation, then reassembled to give a new polytope which (by construction) has equal Ehrhart quasi-polynomial but different $r_P$. 

\begin{example}[$\GL_2(\Z)$-scissor congruence]\label{eg:scissor_congruence_dual_P114}
The lattice triangle $Q:=\sconv{(2, -1),(-1, -1),(-1, 2)}$ with Ehrhart polynomial $L_Q(k)=9/2k^2+9/2k+1$ can be partitioned into two rational triangles as depicted on the left below. Fix the bottom-most triangle, and transform the top-most triangle via the lattice automorphism $e_1\mapsto (3, -1)$, $e_2\mapsto (4,-1)$. This gives the rational triangle $P$ (depicted on the right) from Example~\ref{eg:quasi_period_collapse_dual_P114}.
\begin{center}
\includegraphics[scale=0.8]{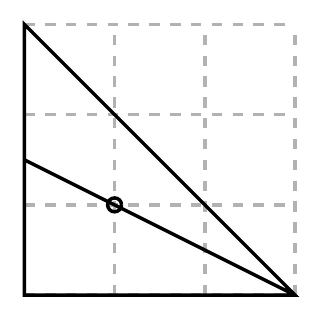}
\raisebox{24px}{$\longmapsto$}
\includegraphics[scale=0.8]{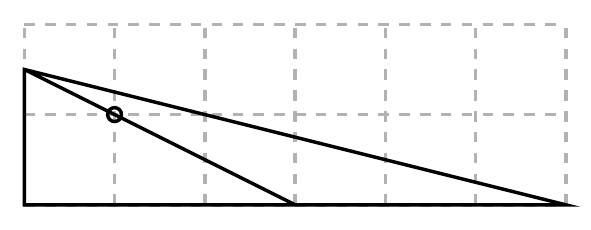}
\end{center}
\end{example}

We give an explanation for quasi-period collapse in two dimensions for a certain class of polygons in terms of recent results in algebraic geometry arising from Mirror Symmetry. In~\S\ref{sec:mutation} we explain how \emph{mutation} -- a combinatorial operation arising from the theory of cluster algebras -- gives an explanation of this phenomenon, and explain how this is related to \emph{$\Q$-Gorenstein (qG-) deformations} of del~Pezzo surfaces as studied by Wahl~\cite{Wah80}, Koll\'ar--Shepherd-Barron~\cite{ksb}, Hacking--Prokhorov~\cite{hp}, and others. Finally, in Corollary~\ref{cor:disc} we completely characterise the discrepancy between the denominator and the quasi-period for this class of polygons.

%-------------------------------------------------------------------------------
\section{Mutation}\label{sec:mutation}
%-------------------------------------------------------------------------------
In~\cite{hm08} Haase--McAllister propose the open problem of finding a systematic and useful technique that implements $\GL_d(\Z)$-scissor congruence for rational polytopes. In the case when the dual polyhedron is a lattice polytope it was observed in~\cite{acgk12} that one such technique is given by \emph{mutation}.

%-------------------------------------------------------------------------------
\subsection{The combinatorics of mutation}\label{subsec:mutation_general}
%-------------------------------------------------------------------------------
Let $N\cong\Z^d$ be a rank $d$ lattice and set $\NQ:=N\otimes_\Z\Q$. Let $P\subset\NQ$ be a lattice polytope. We require -- and will assume from here onwards -- that $P$ satisfies the following two conditions:
\begin{enumerate}[label=(\alph*), ref=\alph*]
\item\label{item:fano_max_dim}
$P$ is of maximum dimension in $N$, $\dim{P}=d$;
\item\label{item:fano_origin}
the origin is contained in the strict interior of $P$, $\orig\in\intr{P}$.
\end{enumerate}
Condition~\eqref{item:fano_origin} is not especially stringent, and can be satisfied by any polytope with $\intr{P}\cap N\neq\emptyset$ by lattice translation. It is, however, an essential requirement in what follows.

Let  $M:=\Hom(N,\Z)\cong\Z^d$ denote the dual lattice. Given a polytope $P\subset\NQ$, the dual polyhedron is defined by
\[
\dual{P}:=\{u\in\MQ\mid u(v)\geq -1\text{ for all }v\in P\}\subset\MQ.
\]
Condition~\eqref{item:fano_origin} gives that $\dual{P}$ is a (typically rational) polytope. It is on rational polytopes dual to lattice polytopes that we focus. In this section we will explain how mutation corresponds to a piecewise-$\GL_d(\Z)$ transformation of $\dual{P}$, and hence is an instance of $\GL_d(\Z)$-scissor congruence for $\dual{P}$.

Following~\cite[\S3]{acgk12}, let $w\in M$ be a primitive lattice vector. Then $w:N\rightarrow\Z$ determines a height function (or grading) which naturally extends to $\NQ\rightarrow\Q$. We call $w(v)$ the \emph{height} of $v\in\NQ$. We denote the set of all points of height $h$ by $H_{w,h}$, and write
\[
w_h(P):=\conv{H_{w,h}\cap P\cap N}\subset\NQ
\]
for the (possibly empty) convex hull of lattice points in $P$ at height $h$. 

\begin{definition}\label{defn:factor}
A \emph{factor} of $P\subset\NQ$ with respect to $w\in M$ is a lattice polytope $F\subset w^\perp$ such that for every negative integer $h\in\Z_{<0}$ there exists a (possibly empty) lattice polytope $R_h\subset\NQ$ such that
\[
H_{w,h}\cap\V{P}\subseteq R_h+\abs{h}F\subseteq w_h(P).
\]
Here `$+$' denotes Minkowski sum, and we define $\emptyset+Q=\emptyset$ for every lattice polytope $Q$.
\end{definition}

\begin{definition}\label{defn:mutation}
Let $P\subset\NQ$ be a lattice polytope with $w\in M$ and $F\subset\NQ$ as above. The \emph{mutation} of $P$ with respect to the data $(w,F)$ is the lattice polytope
\[
\mut_{(w,F)}(P):=\conv{\bigcup_{h\in\Z_{<0}}\!R_h\cup\bigcup_{h\in\Z_{\geq 0}}\!\left(w_h(P)+hF\right)}\subset\NQ.
\]
\end{definition}
It is shown in~\cite[Proposition~1]{acgk12} that, for fixed data $(w,F)$, any choice of $\{R_h\}$ satisfying Definition~\ref{defn:factor} gives $\GL_d(\Z)$-equivalent mutations. Since we regard lattice polytopes as being defined only up to $\GL_d(\Z)$-equivalence, this means that mutation is well-defined. One can readily see that translating the factor $F$ by some lattice point $v\in w^\perp\cap N$ gives isomorphic mutations: $\mut_{(w,F+v)}(P)\cong\mut_{(w,F)}(P)$. In particular if $\dim{F}=0$ then $\mut_{(w,F)}(P)\cong P$. Finally, we note that mutation is always invertible~\cite[Lemma~2]{acgk12}: if $Q:=\mut_{(w,F)}(P)$ then $P=\mut_{(-w,F)}(Q)$.

\begin{remark}
Informally, mutation corresponds to the following operation on slices $w_h(P)$ of $P$: at height $h$ one Minkowski adds or ``subtracts'' $\abs{h}$ copies of $F$, depending on the sign of $h$. Definition~\ref{defn:factor} ensures that the concept of Minkowski subtraction makes sense.
\end{remark}

Mutation has a natural description in terms of the dual polytope $\dual{P}$~\cite[Proposition~4 and pg.~12]{acgk12}.

\begin{definition}\label{defn:inner_normal_fan}
The \emph{inner-normal fan} in $\MQ$ of a polytope $F\subset\NQ$ is generated by the cones
\[
\sigma_{v_F}:=\{u\in\MQ\mid u(v_F)=\min{u(v)\mid v\in F}\},\qquad\text{ for each }v_F\in\V{F}.
\]
\end{definition}

\noindent
A mutation $\mut_{(w,F)}$ induces a piecewise-$\GL_d(\Z)$ transformation $\varphi_{(w,F)}$ on $\MQ$ given by
\[
\varphi_{(w,F)}:u\mapsto u-\umin w,\qquad\text{ where }\umin:=\min{u(v_F)\mid v_F\in\V{F}}.
\]
The inner-normal fan of $F$ determines a chamber decomposition of $\MQ$, and $\varphi_{(w,F)}$ acts linearly within each chamber. Let $Q:=\mut_{(w,F)}(P)$. Then $\varphi_{(w,F)}(\dual{P})=\dual{Q}$. It is clear that the Ehrhart quasi-polynomials $L_{\dual{P}}$ and $L_{\dual{Q}}$ for the dual polytopes are equal, since the map $\varphi_{(w,F)}$ is piecewise-linear. Hence mutation gives a systematic way to produce examples of $\GL_d(\Z)$-scissor congruence.

\begin{example}[Mutation]\label{eg:P2_to_P114}
Let $P=\sconv{(1,0),(0,1),(-1,-1)}\subset\NQ$ and $w=(2,-1)\in M$. Then $F=\sconv{(0,0),(-1,-2)}\subset w^\perp$ is a factor. We see that $Q:=\mut_{(w,F)}(P)=\sconv{(1,0),(0,1),(-1,-4)}$.
\begin{center}
\begin{tabular}{rc@{\ }c@{\ }cr}
\raisebox{92px}{$\NQ:$}&
\raisebox{62.4px}{\includegraphics[scale=0.8]{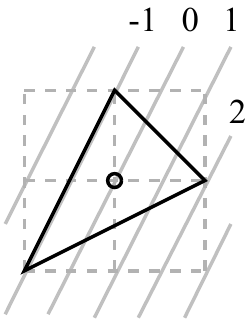}}&%original height = 92px
\raisebox{92px}{$\longmapsto$}&
\includegraphics[scale=0.8]{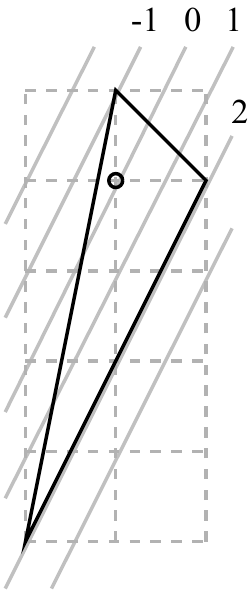}&%original height = 170px
\phantom{$\NQ:$}\\
\end{tabular}
\end{center}
On the dual side we have that $\MQ$ is divided into two chambers whose boundary is given by $\Q\cdot w$, and
\[
\varphi_{(w,F)}:(u_1,u_2)\mapsto
\begin{cases}
(u_1,u_2),&\text{ if }u_1+2u_2\leq 0;\\
(3u_1+4u_2,-u_1-u_2),&\text{ otherwise}.
\end{cases}
\]
\begin{center}
\begin{tabular}{rc@{\ }c@{\ }cr}
\raisebox{27px}{$\MQ:$}&
\includegraphics[scale=0.8]{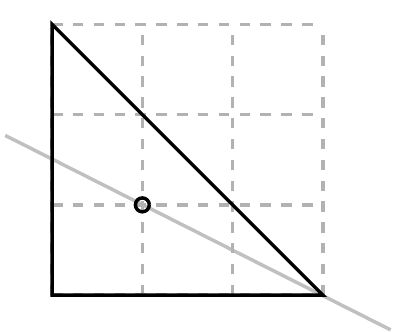}&%original height = 95.5px
\raisebox{27px}{$\longmapsto$}&
\includegraphics[scale=0.8]{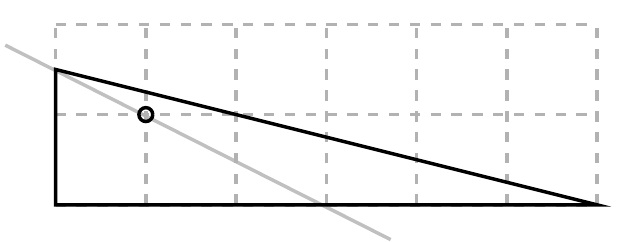}&%original height = 69.5px
\phantom{$\MQ:$}\\
\end{tabular}
\end{center}
Thus we recover Example~\ref{eg:scissor_congruence_dual_P114} from the view-point of mutation.
\end{example}

From here onwards we assume that $P\subset\NQ$ is \emph{Fano}. That is, in addition to  conditions~\eqref{item:fano_max_dim} and~\eqref{item:fano_origin} above, $P$ satisfies:
\begin{enumerate}[label=(\alph*), ref=\alph*, start=3]
\item\label{item:fano_primitive}
the vertices $\V{P}$ of $P$ are primitive lattice points.
\end{enumerate}
The property of being Fano is preserved under mutation~\cite[Proposition~2]{acgk12}. A Fano polytope $P$ corresponds to a toric Fano variety $X_P$ via the \emph{spanning fan} (that is, the fan whose cones are spanned by the faces of $P$). See~\cite{cls} for the theory of toric varieties and~\cite{kn13} for a survey of Fano polytopes. When $P$ is a Fano polygon, $X_P$ corresponds to a toric del~Pezzo surface with at worst log terminal singularities. The \emph{singularity content} of $P$, which we recall in Definition~\ref{def:sing_content} below, is a mutation-invariant of $P$ introduced in~\cite{ak14}. In~\S\ref{subsec:mutation_geometry} we remark briefly on the connection between singularity content and the qG-deformation theory of $X_P$, and how this gives a geometric explanation for the quasi-period collapse of $\dual{P}$.

%-------------------------------------------------------------------------------
\subsection{Quotient singularities}\label{subsec:quot_sing}
%-------------------------------------------------------------------------------
In order to state the definition of singularity content we first recall some of the theory of quotient or orbifold surface singularities. A cyclic quotient singularity is a surface singularity isomorphic to a quotient  $\A^2/G$, where $G$ is a finite cyclic group acting diagonally on $\A^2$. Assuming that $G$ acts faithfully means that it can be expressed as a subgroup of $\GL_2(\C)$ generated by
\[
\begin{pmatrix}
\varepsilon&\\
&\varepsilon^a
\end{pmatrix}
\]
where $\varepsilon$ is a root of unity and $a\in\Z$. Suppose that $G$ has order $r$; all possible representations are obtained (non-uniquely) by letting $a$ range over $0,\dots,r-1$. If $G$ is generated by the matrix above for $\varepsilon$ a primitive $r$-th root of unity, then denote by $\frac{1}{r}(1,a)$ the singularity $\A^2/G$. As a quotient of affine space by an abelian group, $\frac{1}{r}(1,a)$ is an affine toric variety whose fan we now describe.

Let $N\cong\Z^2$ and $M=\Hom_\Z(N,\Z)$ be the cocharacter and character lattices respectively of an algebraic two-torus $(\C^\times)^2$. A cone $\sigma\subset\NQ$ whose rays are generated by lattice points in $N$ describes an affine toric variety $X_\sigma$. More generally, a collection of cones given by a fan $\Sigma$ describes a non-affine toric variety $X_\Sigma$. The singularity $\frac{1}{r}(1,a)$ is the affine toric variety associated to the cone
\[
\sigma=\scone{e_2,re_1-ae_2}\subset\NQ.
\]
The lattice height of such a cone -- that is, the lattice distance between the origin and the line segment joining the two primitive ray generators of the cone (the \emph{edge} of the cone) -- is called the \emph{local index}, and can be calculated to be
\[
\ell_\sigma=\frac{r}{\gcd{r,a+1}}.
\]
The \emph{width} of the cone is the number of unit-length lattice line segments along the edge of the cone or, equivalently, one less than the number of lattice points along the edge. The width is equal to $\gcd{r,a+1}$. We will often conflate a singularity and its corresponding cone in $\NQ$. An isolated cyclic quotient singularity is a $T$\emph{-singularity} if it is smoothable by a qG-deformation.

\begin{lemma}[{\!\cite[Proposition~3.11]{ksb}}]\label{lem:T-sing}
An isolated cyclic quotient singularity is a $T$-singularity if and only if it takes the form
\[
\frac{1}{dn^2}(1,dnc-1)
\]
for some $c$ with $\gcd{n,c}=1$.
\end{lemma}

The cone $\sigma\subset\NQ$ associated to a $T$-singularity $\frac{1}{dn^2}(1,dnc-1)$ has local index $\ell=n$ and width $dn$; it is easily seen that $T$-singularities are characterised by having the width divisible by the local index. Suppose that $P\subset\NQ$ is a Fano polygon with edge $E$ spanning $\sigma$. Let $w\in M$ be the primitive inner-normal such that $w(E)=-\ell$, and choose $F\subset w^\perp$ of lattice length $d$. The mutation $\mut_{(w,F)}(P)$ collapses the edge $E$ to a vertex, removing the cone $\sigma$. This is equivalent to a local qG-smoothing of the $T$-singularity.

\begin{example}
Consider the polytope $Q:=\sconv{(1,0),(0,1),(-1,-4)}$ appearing in Example~\ref{eg:P2_to_P114}. The corresponding spanning fan has three two-dimensional cones, two of which are smooth and one of which, $\scone{(1,0),(-1,-4)}$, corresponds to a $\frac{1}{4}(1,1)$ $T$-singularity.
\end{example}

The other relevant class of quotient singularities are the \emph{$R$-singularities} introduced in~\cite{ak14}.

\begin{definition}
A cyclic quotient singularity of local index $\ell$ and width $k$ is an \emph{$R$-singularity} if $k<\ell$.
\end{definition}

Let $\sigma\subset\NQ$ be a cone of local index $\ell$ and width $k$. Write $k=d\ell+r$, where $d,r\in\Z_{\geq 0}$, $0\leq r<\ell$. If $r=0$ then $\sigma$ is a $T$-singularity. Assume that $r\neq 0$ and, as before, suppose that $P\subset\NQ$ is a Fano polygon with edge $E$ spanning $\sigma$. Let $w\in M$ be the corresponding inner-normal, and pick $F\subset w^\perp$ of lattice length $d$. The mutation $\mut_{(w,F)}(P)$ transforms $\sigma$ to a cone $\tau$ of width $r$ corresponding to a $\frac{1}{r\ell}(1,rc/k-1)$ singularity. Crucially, $\tau$ has width strictly less than the local index, and so cannot be simplified via further mutation. This is equivalent to a partial qG-smoothing of the original singularity $\sigma$, resulting in a singularity $\tau$ that is rigid under qG-deformation. The $R$-singularity $\tau$ is independent of the choices made~\cite[Proposition~2.4]{ak14}.

\begin{definition}
Let $\sigma\subset\NQ$ be a cone corresponding to a $\frac{1}{r}(1,c-1)$ singularity. Let $\ell$ be the local index and let $k$ be the width of the cone. Write $k=d\ell+r$, where $d,r\in\Z_{\geq 0}$, $0\leq r<\ell$. The \emph{residue} of $\sigma$ is
\[
\res{\sigma}=\begin{cases}
\frac{1}{r\ell}(1,rc/k-1),&\text{ if }r\neq 0;\\
\emptyset,&\text{ otherwise.}
\end{cases}
\]
The \emph{singularity content} of $\sigma$ is the pair $(d,\res{\sigma})$. The singularity content is local qG-deformation-theoretic data about $\sigma$.
\end{definition}

\begin{definition}\label{def:sing_content}
Let $P\subset\NQ$ be a Fano polygon with cones $\sigma_1,\ldots,\sigma_n$. The \emph{basket} of $P$ is the multiset
\[
\cB:=\{\res{\sigma_i}\mid 1\leq i\leq n\},
\]
where the empty residues are omitted\footnote{In~\cite{ak14} the basket is cyclically ordered. Although important from the viewpoint of classification, it is not required here.}. The \emph{singularity content} of $P$ is the pair
\[
(d_1+\cdots+d_n,\cB),
\]
where the $d_i$ are the integers appearing in the singularity content of the $\sigma_i$. Singularity content is a qG-deformation-invariant of $X_P$. 
\end{definition}

%-------------------------------------------------------------------------------
\subsection{Hilbert series}
%-------------------------------------------------------------------------------
Any projective toric variety $X_P$ arising from a polytope $P$ comes with a natural ample divisor $D$ given by its toric boundary $D=X_P\setminus\T$, where $\T$ is the big torus inside $X_P$. When $P$ is Fano, $D=-K$, the anti-canonical divisor on $X_P$. In this case, due to the standard toric dictionary allowing one to move between lattice points in $M$ and sections of line bundles on $X_P$, one has that the Hilbert function of $(X_P,-K)$ equals the Ehrhart quasi-polynomial $L_{\dual{P}}(k)$ of the rational polytope $\dual{P}$. Hence the generating function $\Hilb_{(X_P,-K)}(t)$ for the Hilbert function of $(X_P,-K)$ -- the \emph{Hilbert series} of $(X_P,-K)$ -- is equal to the Ehrhart series of $\dual{P}$. From here onwards we supress $-K$ from the notation.

The Hilbert series of an orbifold del~Pezzo surface $X$ with basket $\cB$ can be written in the form~\cite[Corollary~3.5]{ak14}:
\[
\Hilb_X(t)=\frac{1+(K^2-2)t+t^2}{(1-t)^3}+\sum_{\sigma\in\cB}Q_\sigma,
\]
where $Q_\sigma$ are \emph{orbifold correction terms} given by certain rational functions with denominators $1-t^{\ell_\sigma}$. For example, the orbifold correction term for the $R$-singularity $\frac{1}{3}(1,1)$ is
\[
Q_{\frac{1}{3}(1,1)}=\frac{-t}{3(1-t^3)}=-\frac{1}{3}(t+t^4+t^7+\dots)
\]
which contributes $-1/3$ to the coefficient of $t^d$ when $d\equiv 1\modb{3}\ $.

The Hilbert function is a quasi-polynomial when $X$ is an orbifold (because the anti-canonical divisor is $\Q$-Cartier rather than Cartier). The anti-canonical divisor does not correspond to a line bundle, but some integer multiple of it does. The smallest integer $d$ such that $-dK$ is Cartier is called the \emph{Gorenstein index} of $X$ and denoted by $\ell_X$. In the toric setting, $-dK$ is Cartier if and only if $d\dual{P}$ is a lattice polytope. Hence the Gorenstein index $\ell_{X_P}$ of $X_P$ equals the denominator $r_{\dual{P}}$ of $\dual{P}$.

%-------------------------------------------------------------------------------
\subsection{Algebraic geometry and the quasi-period}\label{subsec:mutation_geometry}
%-------------------------------------------------------------------------------
Mutations were introduced in~\cite{acgk12} as part of an ongoing program investigating Mirror Symmetry for Fano manifolds~\cite{ProcECM}. In two dimensions the picture is very well understood: see~\cite{Pragmatic} for the details. In summary, if two Fano polygons $P$ and $Q\subset\NQ$ are related by a sequence of mutations then there exists a qG-deformation between the corresponding toric del~Pezzo surfaces $X_P$ and $X_Q$. Such a qG-deformation preserves the anti-canonical Hilbert series, hence $L_{\dual{P}}=L_{\dual{Q}}$ and so the quasi-periods of $\dual{P}$ and $\dual{Q}$ agree. However it does not in general preserve the Gorenstein index, and hence the denominators $r_{\dual{P}}$ and $r_{\dual{Q}}$ need not be equal. The cones over the edges of $P$ correspond to the singularities of $X_P$, and these admit partial qG-smoothings to the qG-rigid singularities given by the basket $\cB$ of residues.

Suppose that the singularity content of $P$ is $(d,\cB)$. Then, by the absence of global obstructions to qG-deformations on Fano varieties, $X_P$ is qG-deformation-equivalent to a (\emph{not necessarily toric}) del~Pezzo surface $X$ with singularities $\cB$ and whose non-singular locus has topological Euler number $d$. Since $\Hilb_{X_P}(t)=\Hilb_X(t)$, we have an explanation for quasi-period collapse of the dual polytope $\dual{P}$. Specifically, the Gorenstein index of $X$ is equal to the quasi-period of $\dual{P}$.

%-------------------------------------------------------------------------------
\section{Studying quasi-period collapse}
%-------------------------------------------------------------------------------
The Hilbert series of orbifold del~Pezzo surfaces were studied in~\cite{w18} with the aim of describing the structure of the set of possible baskets $\cB$ of $R$-singularities on orbifold del~Pezzo surfaces with a fixed Hilbert series. This is achieved by partitioning $\cB$ into two pieces: a \emph{reduced basket} and an \emph{invisible basket}. The latter, along with the $T$-singularities, is not detectable by the Hilbert series, and from our viewpoint it is this invisibility that causes quasi-period collapse.

\begin{definition}
A collection $\sigma_1,\ldots,\sigma_n$ of $R$-singularities is a \emph{cancelling tuple} if
\[
Q_{\sigma_1}+\cdots+Q_{\sigma_n}=0.
\]
A collection of $R$-singularities is called \emph{invisible} if it is a union of cancelling tuples.
\end{definition}

\begin{example}
Let $\sigma$ be an $R$-singularity of local index $\ell$ and width $k$. Then there exists an $R$-singularity $\sigma'$ of local index $\ell $ and width $\ell-k$ such that $Q_\sigma+Q_{\sigma'}=0$. Combinatorially, this is understood by the observation that the union of the two cones gives a $T$-singularity.
\begin{center}
\vspace{0.7em}
\includegraphics[scale=0.5]{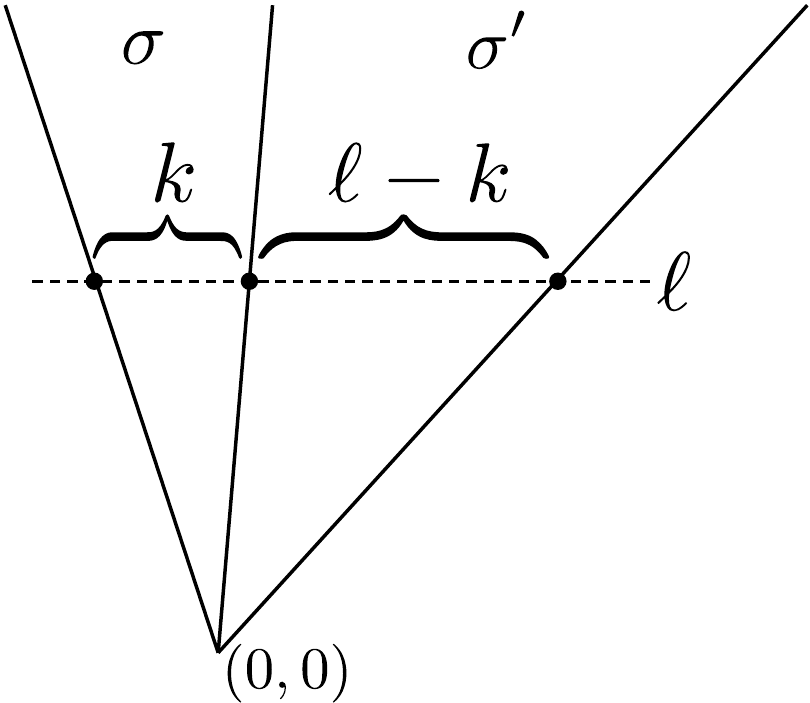}
\end{center}
\end{example}

\begin{definition}
Let $X$ be an orbifold del~Pezzo surface. A maximal invisible subcollection of the basket $\cB$ of $X$ is called an \emph{invisible basket} for $X$. Notice that such a maximal subcollection is not unique, since singularities can appear in many different cancelling tuples. Given a choice of invisible basket $\cIB\subset\cB$, the complement $\cRB=\cB\setminus\cIB$ is called the \emph{reduced basket} for $X$ corresponding to the choice of $\cIB$.
\end{definition}

Let $P\subset\NQ$ be a Fano polygon with singularity content $(d,\cB)$. Let $\cIB$ be an invisible basket of $\cB$ with corresponding reduced basket $\cRB$. Hence $\cB=\cRB\amalg\cIB$. Denote the collection of $T$-singularities on $X_P$ by $\cT$ (so $\abs{\cT}=d$).

\begin{thm}\label{thm:measuring_qp}
Let $P\subset\NQ$ be a Fano polygon. The quasi-period of $\dual{P}$ is given by\footnote{We adopt the convention that $\lcm{\emptyset}=1$.}
\[
\pi_{\dual{P}}=\lcm{\ell_\sigma\mid \sigma\in\cRB}.
\]
Furthermore, $\dual{P}$ exhibits quasi-period collapse if and only if there exists some $\tau\in\cIB\cup\cT$ of local index not dividing $\lcm{\ell_\sigma\mid \sigma\in\cRB}$. Moreover, the quasi-period collapse is measured by $\cIB$:
\[
r_{\dual{P}}=\blcm{\{\pi_{\dual{P}}\}\cup\{\ell_\sigma\mid\sigma\in\cIB\cup\cT\}}.
\]
\end{thm}
\begin{proof} We have
\[
\Ehr_{\dual{P}}(t)=\Hilb_{X_P}(t)=\text{initial term}+\sum_{\sigma\in\cB}{Q_\sigma}=\text{initial term}+\sum_{\sigma\in\cRB}{Q_\sigma}.
\]
As discussed, each orbifold correction term $Q_\sigma$ contributes to the coefficients of this series as a quasi-polynomial with quasi-period $\ell_\sigma$. When $\sigma\in\cRB$ these terms are not cancelled and so make non-zero contributions to the coefficients of the Ehrhart series, hence its quasi-period is given by:
\[
\pi_{\dual{P}}=\lcm{\ell_\sigma\mid\sigma\in\cRB}.
\]
The Gorenstein index of $P$ is equal to $\ell_{X_P}=\lcm{\ell_\sigma\mid\sigma\in\cB\cup\cT}$. Hence
\[
r_{\dual{P}}=\ell_{X_P}=\lcm{\ell_\sigma\mid\sigma\in\cRB\cup\cIB\cup\cT}=\blcm{\{\pi_{\dual{P}}\}\cup\{\ell_\sigma\mid\sigma\in\cIB\cup\cT\}}.
\]
This is distinct from $\pi_{\dual{P}}$ if and only if $\lcm{\ell_\sigma\mid\sigma\in\cIB\cup\cT}$ does not divide $\pi_{\dual{P}}$.
\end{proof}

\begin{remark}
It follows from~\cite[\S4]{w18} that the choice of $\cIB$ is irrelevant in the statement of Theorem~\ref{thm:measuring_qp}.
\end{remark}

As a corollary to Theorem~\ref{thm:measuring_qp} we immediately obtain:

\begin{cor}\label{cor:disc}
Let $P\subset\NQ$ be a Fano polygon. The discrepancy between the quasi-period and denominator of $\dual{P}$ is
\[
\frac{r_{\dual{P}}}{\pi_{\dual{P}}}=\frac{\lcm{\ell_\sigma\mid\sigma\in\cIB\cup\cT}}{\gcd{\lcm{\ell_\sigma\mid\sigma\in\cRB},\lcm{\ell_\sigma\mid\sigma\in\cIB\cup\cT}}}.
\]
\end{cor}

\begin{example}[Detecting quasi-period collapse]
Consider the polytope $Q:=\sconv{(1,0),(0,1),(-1,-4)}$ appearing in Example~\ref{eg:P2_to_P114}. This has singularity content $(3,\emptyset)$, and $\cT=\{2\times\text{smooth},\frac{1}{4}(1,1)\}$. Applying Corollary~\ref{cor:disc} we have that $r_{\dual{Q}}=2\pi_{\dual{Q}}$.
\end{example}

We now give an example of an infinite family of Fano triangles, obtained via mutation, where the denominator $r_{\dual{P}}$ can become arbitrarily large but where $\pi_{\dual{P}}=1$. Let $P\subset\NQ$ be a Fano triangle. Recall that the corresponding toric variety $X_P$ is a \emph{fake weighted projective plane}~\cite{Kas09}: a quotient of a weighted projective plane by a finite group $N/N'$ acting free in codimension one, where $N'$ is the sublattice generated by the vertices of $P$.

\begin{example}[Mutations of $\Proj^2$]\label{eg:arbitrarily_large_denominator}
In~\cite{hp,ak16} the graph of mutations of $\Proj^2$ is constructed. The vertices of this graph are given by $\Proj(a^2,b^2,c^2)$, where $(a,b,c)\in\Z_{>0}^3$ is a \emph{Markov triple} satisfying
\begin{equation}\label{eq:markov}
a^2+b^2+c^2=3abc.
\end{equation}
Let $X_P=\Proj(a^2,b^2,c^2)$ be such a weighted projective plane, with $P\subset\NQ$ the corresponding Fano triangle. Since $X_P$ is qG-deformation-equivalent to $\Proj^2$, so $X_P$ is smoothable and its anti-canonical Hilbert function has quasi-period one. Hence $\pi_{\dual{P}}=1$. However, the denominator $r_{\dual{P}}$ of $\dual{P}$ can be arbitrarily large. To see this, note first that $a,b,c$ must be pairwise coprime: if $p\mid a$ and $p\mid b$ then $p^2\mid 3abc=a^2+b^2+c^2$, and hence $p\mid c$; but then $p$ appears as a square on the left-hand side and as a cube on the right-hand side of~\eqref{eq:markov}. Let $\ol{b}$ be an inverse of $b\!\!\modb{a^2}$. Note that $c^2\ol{b}^2+1\equiv(3abc-b^2)\ol{b}^2+1\equiv3a\ol{b}c\modb{a^2}$, and so the singularity $\frac{1}{a^2}(b^2,c^2)$ on $X_P$ has local index
\[
\frac{a^2}{\gcd{a^2,c^2\ol{b}^2+1}}=\begin{cases}
a,&\text{ if }a\not\equiv0\modb{3}; \\
a/3,&\text{ if }a\equiv0\modb{3}.
\end{cases}
\]
Considering equation~\eqref{eq:markov} $\!\!\modb{3}$ shows that no Markov numbers are divisible by three. Hence the three local indices on $X_P$ are $a$, $b$, and $c$, and so $r_{\dual{P}}=abc$. The two triangles $P$ and $Q$ in Example~\ref{eg:P2_to_P114} are the simplest examples, arising from the Markov triples $(1,1,1)$ and $(1,1,2)$ respectively, and corresponding to $\Proj^2$ and $\Proj(1,1,4)$.
\end{example}

\begin{remark}
There exist Fano triangles of quasi-period one not arising from the construction in Example~\ref{eg:arbitrarily_large_denominator}. For example, consider
\[
P=\sconv{(3,2),(-1,2),(-1,-2)}\subset\NQ.
\]
The corresponding fake weighted projective plane $X_P=\Proj(1,1,2)/(\Z/4)$ has $2\times\frac{1}{4}(1,3)$ and $\frac{1}{8}(1,3)$ $T$-singularities. We see that $\dual{P}$ has $r_{\dual{P}}=2$ and $\pi_{\dual{P}}=1$. In fact $X_P$ is qG-smoothable to the nonsingular del~Pezzo surface of degree two, and hence $L_{\dual{P}}(k)=k^2+k+1$.
\end{remark}

%-------------------------------------------------------------------------------
\subsection*{Acknowledgments}
%-------------------------------------------------------------------------------
Our thanks to Tom Coates, Alessio Corti, Mohammad Akhtar, Thomas Prince, and Miles Reid for many helpful discussions. AK is supported by EPSRC Fellowship~EP/N022513/1. BW was supported by a grant from the London Mathematical Society.
%-------------------------------------------------------------------------------

\bibliographystyle{plain}
\bibliography{bibliography}
%-------------------------------------------------------------------------------
\end{document}